\documentclass{commat}

\title{%
	On biharmonic functions on vector bundles
}

\author{%
	Mohamed Tahar Kadaoui Abbassi, Souhail Doua and Ibrahim Lakrini
}

\authorinfo[%
M.T.K. Abbassi]{
	Department of Mathematics, Faculty of Sciences Dhar El Mahraz, Sidi Mohamed Ben Abdellah University, B.P.1796- Atlas, Fez 30000, Morocco}{%
	mtk\_abbassi@yahoo.fr
}

\authorinfo[%
S. Doua]{
	Faculty of Sciences Dhar El Mahraz, Sidi Mohamed Ben Abdellah University, B.P.1796- Atlas, Fez 30000, Morocco}{%
	souherukun@gmail.com
}

\authorinfo[%
I. Lakrini]{
	Department of Mathematics, Faculty of Sciences Dhar El Mahraz, Sidi Mohamed Ben Abdellah University, B.P.1796- Atlas, Fez 30000, Morocco}{%
	lakrinii@gmail.com
}

\abstract{%
	This paper is devoted to the investigation of harmonic and biharmonic functions on vector bundles equipped with spherically symmetric metrics. We will study the biharmonicity of vertical lifts of functions as well as $r$-radial functions on vector bundle manifolds and, as a consequence, we will construct an infinite two-parameter family of proper biharmonic functions. 
}

\keywords{%
	Vector bundle, spherically symmetric metric, Sasaki metric, harmonic function, biharmonic function.
}

\msc{%
	53C07, 53C20, 53C21, 53C24, 53C43.
}

\VOLUME{32}
\NUMBER{1}
\YEAR{2024}
\firstpage{261}
\DOI{https://doi.org/10.46298/cm.11699}

\begin{document}
	\allowdisplaybreaks
	
	
	\section*{Introduction}
	
	
	The study of harmonic and biharmonic functions constitutes an important problem in geometric analysis. To be precise, function theory on certain manifolds provides several insights into their geometries (see \cite{CM1,CM2,CM3,CM4,Y1,CY1}). Surprisingly, in spite of their importance, examples of harmonic and biharmonic functions are scarce, and examples of proper biharmonic functions are even scarcer. Even for the general theory of harmonic mappings, few examples of (proper) (bi-)harmonic mappings do exist. Interesting such examples are given by (proper) (bi-)harmonic vector fields and functions on tangent and unit tangent sphere bundles endowed with $g$-natural metrics (see \cite{Abb-Calv,Abb-Calv-Per1,Abb-Calv-Per2,Abb-Calv-Per3,Abb-Dou} and the references therein). In this paper, we are interested in finding examples of (proper) biharmonic functions.  
	
	In the context of geometry of vector bundles, the first and third authors studied harmonic function theory for a class of Riemannian metrics known as \emph{spherically symmetric metrics} (cf. \cite{Abb-Lak5}). They also used this to infer interesting results on the geometry of these spaces.  In particular, they studied the harmonicity of vertical lifts as well as a class of functions known as $r$-radial functions \cite{Abb-Lak5}. We will investigate, in this paper, the biharmonicity of these classes and we will explicitly construct an infinite two-parameter family of examples of proper biharmonic functions on vector bundle manifolds endowed with spherically symmetric metrics.
	
	The paper is organized as follows. The first section mainly serves as a background for the two sections to follow. It contains the necessary preliminaries on harmonic and biharmonic functions as well as spherically symmetric metrics on vector bundles.
	
	In the second section, we study the biharmonicity of functions vertically lifted to vector bundles when equipped with spherically symmetric metrics. We compute the bilaplacian of the vertical lift of a smooth function (Proposition \ref{bilaplacian_vertical}) and we derive a characterization of spherically symmetric metrics, with respect to which vertical lifts of (proper) biharmonic functions are (proper) biharmonic, allowing us to construct explicit examples (Proposition \ref{exp_vertical}). 
	
	In the third section, we investigate $r-$radial functions: we compute the bilaplacian (Proposition \ref{Bihar-radial}) and we reduce the complexity of the fourth order linear differential equation we obtain by studying the case of the Sasaki metric (Proposition \ref{Sasaki_radial} and Corollary \ref{Sasaki_radial2}). Consequently, we construct two infinite families of examples (Proposition \ref{Radial_Examples}) of biharmonic functions.
	
	All manifolds and geometric objects are assumed to be smooth and smooth will always mean differentiability of class $C^{\infty}$. Unless otherwise stated, all manifolds are assumed to be connected and without boundary. 
	
	
	\section{Preliminaries}
	
	
	\subsection{Harmonic and biharmonic functions}
	
	
	Let $(M,g)$ and $(N,h)$ be two Riemannian manifolds. Denote by $\nabla^{g}$ and $\nabla^{h}$ their Levi-Civita connections, and $m$ and $n$ their dimensions, respectively. We denote by $C^{\infty}(M,N)$ the set of smooth mappings from $M$ to $N$.
	
	For a smooth map $\phi:M\longrightarrow N$, the differential map   $d\phi:TM\longrightarrow TN$, of the map $\phi$, is a smooth section of the bundle $T^{*}M\otimes \phi^{-1}TN$, where $\phi^{-1}TN$ is the pullback of the vector bundle $TN$ by $\phi$. The \emph{energy density} of $\phi$ is the function defined by  $e(\phi)=\frac{1}{2}\|d\phi\|^{2}$, where $\|.\|$ is the \emph{Hilbert-Schmidt norm} of $\phi$. The energy density is locally expressed as
	\[e(\phi)=\frac{1}{2}tr_{g}(\phi^{*}h)=\frac{1}{2}\sum_{i=1}^{m}\phi^{*}h(e_{i},e_{i}),\]
	where $\{e_{i}:\: i=1,\ldots,m\}$ is any local orthonormal frame of $M$.
	
	For any compact subset $\Omega\subset M$, the energy of $\phi$ on $\Omega$ is given by 
	\begin{equation}
		E(\phi,\Omega)=\int_{\Omega}e(\phi)\upsilon_{g},
	\end{equation} 
	where $\upsilon_{g}$ is the Riemannian \emph{volume element} of the Riemannian manifold $(M,g)$.
	
	A \emph{smooth variation} of $\phi$ is a family of mappings $\{\phi_{t}\}_{-\epsilon<t<\epsilon}\subset C^{\infty}(M,N)$ which depends smoothly on $t$ such that 
	the mapping $F:(-\epsilon,\epsilon)\times M\longrightarrow N$ defined by $F(t,x)=\phi_{t}(x)$, for $-\epsilon<t<\epsilon$ and $x\in M$, satisfies 
	\begin{itemize}
		\item [i)]  $F$ is smooth,
		\item [ii)] $F(0,x)=\phi(x),$ for all $x\in M$.
	\end{itemize}
	Any smooth variation $\{\phi_{t}\}_{-\epsilon<t<\epsilon}$ of $\phi$ determines a smooth vector field along $\phi$, that is, a section of the vector bundle $\phi^{-1}TN$, defined by 
	\begin{equation}
		V(x)=\frac{d}{dt}\Big\arrowvert_{t=0}\phi_{t}(x),
	\end{equation}  
	for all $x\in M$. The vector field $V$ is called the \emph{variational vector field} of the variation $\{\phi_{t}\}$. 
	
	A smooth map $\phi:(M,g)\longrightarrow (N,h)$ is said to be a \emph{harmonic map} if, for every compact subset $\Omega$ of $M$, $\phi$ is a critical point of the energy functional $E(.,\Omega)$, i.e., for every compact subset $\Omega$ of $M$ and every variation $\{\phi_{t}\}_{-\epsilon<t<\epsilon}$ of $\phi$, such that $\phi=\phi_{0}$ and $\phi(x)=\phi_{t}(x)$ for all $x\in M\setminus \Omega$ and all $t$, we have 
	\begin{equation*}
		\frac{d}{dt}\Big\arrowvert_{t=0}E(\phi_{t},\Omega)=0.
	\end{equation*} 
	
	If $M$ is a compact manifold and  $\phi:M\longrightarrow N$ is a map, then we denote the energy $E(\phi,M)$ of $\phi$ on $M$ by $E(\phi)$ and, for every smooth variation $\{\phi_{t}\}_{-\epsilon<t<\epsilon}$ of $\phi$, we have
	$$\frac{d}{dt}\Big\arrowvert_{t=0}E(\phi_{t})=-\int_{M}h(V,\tau(\phi))\upsilon_{g},$$
	where $V$ is the variational vector field associated with the variation $\{\phi_{t}\}$ and $\tau(\phi)$ is a vector filed along $\phi$, called the tension field of $\phi$, and given by 
	$$\tau(\phi)=\sum_{i=1}^{m}\big(\phi^{*}\nabla^{h}_{e_{i}}d\phi(e_{i})-d\phi(\nabla^{g}_{e_{i}}e_{i})\big),$$
	where $\{e_{i}\}$ is a linear orthonormal frame of $M$.
	
	Hence, harmonic maps on compact manifolds are characterized
	as smooth maps with vanishing tension field. Since the equation $\tau(\phi)=0$ is tensorial, it can be generalized to non-compact manifolds. Indeed, a map $\phi:(M,g)\longrightarrow (N,h)$ is said to be harmonic if $\tau(\phi)=0$.
	For further details about harmonic maps, we refer to \cite{report}.
	
	Further, if $N=\mathbb{R}$, then for a function $\phi:M\longrightarrow \mathbb{R}$, the Euler-Lagrange equation $\tau(\phi)=0$ is nothing but the Laplace equation $\Delta_gf=0$, where $\Delta_g$ is the Laplace-Beltrami operator of the Riemannian manifold $(M,g)$. In this case, harmonic maps are just harmonic functions.  
	
	In \cite{Ee.se}, J. Eells and L. Lemaire extended  the notion of harmonic maps to \emph{biharmonic maps} on compact manifolds which are, by definition, critical points of the \emph{bienergy functional}:
	$$E_{2}(\phi)=\frac{1}{2}\int_{M} ||\tau(\phi)||^{2} \upsilon_{g}.$$
	G. Jiang (\cite{Ji}) derived the associated Euler-Lagrange equation of $E_{2}$,
	\begin{equation}\label{bitension_field}
		\widehat{\Delta}\tau(\phi)-\underset{i=1}{\overset{m}{\sum}}R^{N}\big(\tau(\phi),\phi_{*}e_{i}\big)\phi_{*}e_{i}=0,
	\end{equation}
	where $R^{N}$ is the Riemannian curvature tensor of $(N,h)$ and $\widehat{\Delta}$ is the rough Laplacian defined by: $\widehat{\Delta}=-\underset{i=1}{\overset{m}{\sum}}\big(\phi^*\nabla_{e_{i}}\phi^*\nabla_{e_{i}}-\phi^*\nabla_{\nabla_{e_{i}}e_{i}}\big)$, and $\phi^*\nabla$ is the induced connection on $\phi^{-1}TN$.\par
	
	The quantity 
	\[\tau_{2}(\phi):=\widehat{\Delta}\tau(\phi)-\underset{i=1}{\overset{m}{\sum}}R^{N}\big(\tau(\phi),\phi_{*}e_{i}\big)\phi_{*}e_{i}\]
	is called the \emph{bitension field of $\phi$}. As for the case of harmonic maps, when $(M,g)$ is a general Riemannian manifold (including the non-compact case), a map $\phi: (M,g) \rightarrow (N,h)$ is said to be \emph{biharmonic} if $\tau_2 (\phi)=0$.
	
	Furthermore, when $N=\mathbb{R}$ and $f:M\longrightarrow \mathbb{R}$ is a smooth function, then $\tau_2 (f)=\Delta_g^2(f)$, where $\Delta_g^2$ is the bilaplacian operator of $(M,g)$. The function $f$ is said to be a biharmonic function if $\Delta_g^2f=0.$
	
	
	\subsection{Vector bundles and spherically symmetric metrics}
	
	
	Let $(E,\pi, M)$ be a vector bundle of rank $k\geq 1$ with Riemannian base $(M,g)$. We assume that $E$ is endowed with a fiber metric $h$ and a compatible connection $D$
	(i.e., $Dh=0$). Let $K$ denote the connection map associated with $D$. Denote by $\mathcal{V}$ the \emph{vertical subbundle}, that is, the subbundle given by $$\mathcal{V}=\bigcup_{e\in E}V_{e}E,$$
	where, for all $e\in E$, $V_{e}E=\ker(d_{e}\pi)$ is the \emph{vertical subspace} of $T_{e}E$. The elements of $\mathcal{V}$ are called \emph{vertical vectors} and a vector field on $E$ is said to be a \emph{vertical vector field} if it lies completely in $\mathcal{V}$.
	
	Moreover, if $x\in M$ and $e, u\in \pi^{-1}(x)$, set $\gamma_{e,u}(t)=e+tu$ with $t\in (-\epsilon,\epsilon)$, then $\gamma_{e,u}(t)\in \pi^{-1}(x)$, for all $|t|<\epsilon$, then the \emph{vertical lift} of $u$ at $e$ is defined to be the vertical vector given by $u^{v}_{e}=\frac{d}{dt}\big\lvert_{t=0}\gamma_{e,u}(t)$. 
	Using vertical lifts of elements of $E$ to vertical vectors, we can lift sections of $E$. Indeed, if $\sigma$ is a section of $E$, its \emph{vertical lift} is the vertical vector field $\sigma^{v}$ defined by
	\begin{equation*}
		\sigma^{v}(e)=\sigma(\pi(e))^{v}_{e},
	\end{equation*}
	where $\sigma(\pi(e))^{v}_{e}$ is the vertical lift of $\sigma(\pi(e))$ at $e$. The vertical subbundle $\mathcal{V}$ is naturally isomorphic to the pullback of $E\overset{\pi}{\longrightarrow} M$ by $\pi$ which we denote by $\pi^{\star}E\longrightarrow E$ (see \cite{Dra-Per,Poor}).
	
	The connection $D$ gives rise, at each point $e\in E$, to the \emph{horizontal subspace} given by $H_{e}E=\ker{(K_{e})}$, which is complementary to the vertical subspace, i.e., $T_{e}E=V_{e}E\oplus H_{e}E$, which gives the \emph{horizontal subbundle} $$\mathcal{H}=\bigcup_{e\in E}H_{e}E,$$ such that  
	\begin{equation}\label{TE-Decomp}
		TE=\mathcal{H}\oplus \mathcal{V}.
	\end{equation}
	
	Elements of $\mathcal{H}$ are called \emph{horizontal vectors}, and a vector field is said to be a \emph{horizontal vector field} if it lies completely in $\mathcal{H}$. If $x\in M$ and $X\in T_{x}M$, the \emph{horizontal lift} of $X$ at $e$ is the horizontal vector $X^{h}_{e}$ satisfying $d_{e}\pi(X^{h}_{e})=X$. Further, if $X\in \mathfrak{X}(M)$, then the horizontal vector lift of $X$ is the horizontal vector field $X^{h}\in \mathfrak{X}(E)$ such that $(X^{h})_{e}$ is the horizontal lift of $X_{\pi(x)}$ at $e$. 
	The horizontal sub-bundle is naturally isomorphic to the pullback vector bundle $\pi^{*}TM\longrightarrow E$, thus $$\mathcal{H}\oplus \mathcal{V}=TE\simeq \pi^{*}TM\oplus \pi^{\star}E$$
	
	The later splitting gives rise to a splitting of vector fields of $E$. Indeed, for all $Z\in \mathfrak{X}(E)$, we have $Z=Z^{H}+Z^{V}$, where $Z^{H}$ is the horizontal component and $Z^{V}$ is the vertical component.
	
	In the context of Riemannian geometry of the tangent bundle of a Riemannian manifold, S. Sasaki introduced, in \cite{Sas1}, a Riemannian metric induced from the Riemannian metric on the base manifold known as the \emph{Sasaki metric}, we refer the reader to \cite{Abb1,Abb2} for details on the Riemannian geometry of tangent bundles. 
	It had been shown that the Sasaki metric presents a strong rigidity in the sense of \cite{Kow1,Mus-Tri}. In order to overcome such a rigidity, many classes of metrics were introduced (e.g., \cite{Ben-Lou-W, Mus-Tri}). 
	
	In the case of tangent bundles, the Sasaki metric can be easily expressed as follows. If $(M,g)$ is an $n$-dimensional Riemannian manifold, the Sasaki metric $G^s$ on the tangent bundle $TM$ of $M$ is given by:
	\begin{itemize}
		\item$G^s(X^{h},Y^{h})=g(X,Y)$
		\item$G^s(X^{h},Y^{v})=0$
		\item$G^s(X^{v},Y^{v})=g(X,Y)$
	\end{itemize}
	for all vector fields $X, Y$ on $M$.
	
	A more general class of Riemannian metrics on the total spaces of vector bundles, the so-called \emph{spherically symmetric metrics}, was introduced in \cite{Alb1}. Those metrics are defined as follows.
	\begin{equation}\label{SSM}
		G_{e}(X,Y)=e^{2\varphi_{1}}g(d\pi(X),d\pi(Y))+e^{2\varphi_{2}}h(K_{e}(X),K_{e}(Y)),
	\end{equation}
	where $\varphi_{1}$ and $\varphi_{2}$ are scalar functions on $E$ depending only on the squared norm of $e$, that is, $r=\|e\|^{2}=h(e,e)$. We require the $\varphi_{1}$ an $\varphi_{2}$, as well as all their successive derivatives, to be smooth on $(0,+\infty)$ and at $r=0$ on the right. Spherically symmetric metrics may be expressed using pullbacks as follows. 
	\begin{equation}\label{SSM1}
		G=e^{2\varphi_{1}}\pi^{*}g\oplus e^{2\varphi_{2}}\pi^{\star}h,
	\end{equation}
	
	In the case, $\varphi_1=\varphi_2=0$, we obtain the generalization of the Sasaki metric to vector bundle manifolds. Further, if $E=TM$, we obtain the Sasaki metric itself.
	
	Spherically symmetric metrics have been extensively studied in \cite{Alb1,Abb-Lak1,Abb-Lak2,Abb-Lak3,Abb-Lak5}. We shall recall some necessary technical tools we will be using in the forthcoming calculations. 
	
	In what follows, we will make use of the following notations:
	\begin{itemize}
		\item [(i)] $\langle .,. \rangle_{M}=\pi^{*}g$,
		\item [(ii)] $\langle .,. \rangle_{E}=\pi^{\star}h$,
		\item [(iii)]$\langle .,. \rangle=\pi^{*}g\oplus \pi^{\star}h$.
	\end{itemize}
	
	The Levi-Civita connection of the metric given in (\ref{SSM}) is computed in \cite{Alb1} and given as follows:
	
	\begin{proposition}\label{Levi-Civita}
		Let $X$ and $Y$ be two vector fields on $E$, then
		\begin{equation*}
			\tilde{\nabla}_{X}Y=\tilde{D}_{X}Y+C_{X}Y+A_{X}Y-\frac{1}{2}\mathcal{R}^{\xi}(X,Y),
		\end{equation*}
		where
		\begin{itemize}
			\item [i)]\: $\tilde{D}=\pi^{*}\nabla\oplus \pi^{\star}D$, where $\nabla$ is the Levi-Civita connection of $(M,g)$;
			\item [ii)] $C(.,.)$ is the vector valued form given by
			\begin{equation*}
				\begin{split}
					C_{X}Y=a\big(\xi^{\flat}(X)Y^{H}&+\xi^{\flat}(Y)X^{H}\big)+c_{1}\langle X^{H},Y^{H}\rangle \xi +c_{2}\langle X^{V},Y^{V}\rangle \xi\\ &+b\big(\xi^{\flat}(X)Y^{V}+\xi^{\flat}(Y)X^{V}\big),
				\end{split}
			\end{equation*}
			with
			$$\left\{
			\begin{array}{ll}
				a= 2\varphi_{1}', \qquad c_{1}=-2\varphi_{1}'e^{2(\varphi_{1}-\varphi_{2})}, \\
				b=2\varphi_{2}' , \qquad c_{2}=-2\varphi_{2}',
			\end{array}
			\right.
			$$
			for all $X$, $Y \in \mathfrak{X}(E)$, $\xi$ being the tautological section of the vertical sub-bundle defined by $\xi_{e}=e^{v}_{e}\in \pi^{\star}E$ and where ` $\flat$' is taken with respect to $h$;
			\item [iii)] \: $\mathcal{R}^{\xi}(X,Y)=\pi^{\star}R^{E}(X,Y)\xi$, for all $X$, $Y \in \mathfrak{X}(E)$ where $R^{E}$ is the curvature of $D$;
			\item [iv)] $A(.,.)$ is the $\pi^{*}TM$-valued form defined by
			$$e^{2\varphi_{1}}\langle A(X,Y),Z\rangle_{M}=\frac{e^{2\varphi_{2}}}{2}\left(\langle \mathcal{R}^{\xi}(X,Z),Y\rangle_{E}+\langle \mathcal{R}^{\xi}(Y,Z),X\rangle_{E}\right),$$
			for all $X$, $Y$ and $Z \in \mathfrak{X}(E)$.
		\end{itemize}
	\end{proposition}
	
	
	\subsection{Some classes of functions on vector bundles}
	
	
	Total spaces of vector bundles, and particularly tangent bundles, can be endowed with different classes of functions. We shall focus on lifted functions, that is, functions on total spaces constructed from functions on the base manifold.
	
	In the general context of vector bundles the first and third authors studied, in \cite{Abb-Lak5}, the harmonicity of the following classes of functions:
	\begin{itemize}
		\item [$(\mathcal{C}_1)$] Vertical lifts of functions on $M$, that is, functions of the form $f^{v}=f\circ \pi$ where $f\in C^{\infty}(M)$.
		\item [$(\mathcal{C}_2)$] Functions of the form $F:E\longrightarrow \mathbb{R}$ such that $F(e)=\alpha(r)$, where $r=h(e,e)$ and $\alpha:\mathbb{R}^{+}\longrightarrow \mathbb{R}$ is a real valued function which is smooth on $(0,+\infty)$ and at $r=0$ on the right. We refer to those functions as $r$-radial functions.
		\item [$(\mathcal{C}_3)$] Functions of the form $F_{\sigma}:E\longrightarrow\mathbb{R}$ such that 
		$F_{\sigma}(e)=h(\sigma(\pi(e)),e)$, where $\sigma\in \Gamma(E)$ is a section.
	\end{itemize}
	
	Henceforward, we shall always endow the total space of a vector bundle with a spherically symmetric metric $G$, with weight functions $\varphi_{1}$ and $\varphi_{2}$, of the form \eqref{SSM} (or equivalently, of the form \eqref{SSM1}).
	
	In what follows, we investigate biharmonic functions on total spaces of vector bundles endowed with spherically symmetric metrics. So, we assume that $(E,\pi,M)$ is a vector bundle for which $E$ is endowed with a spherically symmetric metrics of the form \eqref{SSM1}. We shall begin with vertical lifts then we consider $r$-radial functions.
	

	\section{Biharmonicity of vertical lifts of functions to vector bundles}
	

	We shall make calculations in a particularly suitable orthonormal frame which we construct as follows. Let $e \in E$ with $x=\pi(e)$, and let $\{e_i:\: i=1,\ldots,m\}$ be a local linear orthonormal frame of $M$ in a neighbourhood of $x$. For $i=1,\ldots,m$, we denote by $e_i^h$ the horizontal lift of $e_i$ to $E$ with respect to the connection $D$. On the other hand, let $\{\sigma_p:\:p=1,\ldots,k\}$ be an orthonormal frame of $E$ and let $\sigma_p^v$, for $p=1,\ldots,k$, be the vertical lift of $\sigma_{p}$. Then set 
	\begin{equation}\label{Basis}
		E_i=e^{-\varphi_1}e_i^h \quad \textup{ and } \quad E_{m+p}=e^{-\varphi_2}\sigma_p^v, \quad i=1,\ldots,m,\; p=1,\ldots,k;
	\end{equation}
	therefore $\{E_I:\: I=1,\ldots,m+k\}$ is a local linear orthonormal frame of $E$ in a neighbourhood of $e$. We refer to such frames as the \emph{induced local frames} on $E$.
	
	For the sake of completeness, we recall the following lemma (see \cite{Alb1}) which gives the directional derivatives of functions depending on $r$ in the direction of horizontal and vertical vectors:   
	\begin{lemma}\label{Der-fun}
		Let $\alpha$ be a smooth real scalar function. Then, for any horizontal (resp. vertical) vector $X^H$ (resp. $Y^V$) on $E$, we have
		\begin{itemize}
			\item [i)] $X^H(\alpha(r))=0$;
			\item [ii)] $Y^V(\alpha(r))=2\alpha^\prime(r)\xi^{\flat}(Y^V)$.
		\end{itemize}
	\end{lemma}
	
	We will also be using the following formulas for the directional derivatives of the vertical lift of a function in the direction of horizontal and vertical vector fields:
	\begin{itemize}
		\item [$\bullet$] For every $f\in \mathcal{C}^\infty(M)$ and every vertical vector field $X^V$, we have 
		$X^V(f^v)=0$.
		\item [$\bullet$] For every $f\in \mathcal{C}^\infty(M)$ and every vector field $X\in \mathfrak{X}(M)$, we have 
		$X^h(f^v)=(Xf)^v$.
	\end{itemize}
	
	Using induced local frames of the form (\ref{Basis}) and Lemma \ref{Der-fun}, in \cite{Abb-Lak5}, the authors proved the following formulas:
	\begin{itemize}
		\item [$\bullet$] For every $f\in \mathcal{C}^\infty(M)$, we have   $\tilde{\nabla}f^{v}=e^{-2\varphi_{1}}(\nabla f)^{h}.$ 
		\item [$\bullet$] For every $f\in \mathcal{C}^\infty(M)$, we have $\Delta_{G}(f^{v})=e^{-2\varphi_{1}}(\Delta_{g} f)^{v}.$
		\item [$\bullet$] For every  $r$-radial function $F(e)=\alpha(r)$, we have 
		$\tilde{\nabla}F=2e^{-2\varphi_{2}}\alpha' \xi,$ where $\xi$ is the tautological vector field defined earlier. 
		\item [$\bullet$] We have that $div(\xi)=2mr\varphi_{1}'+(1+2r\varphi_{2}')k$.
	\end{itemize}
	
	Using all the previous formulas, we have:
	\begin{proposition} \label{bilaplacian_vertical}
		For every smooth function $f$ on $M$, we have 
		\begin{align*}
			\Delta_G^2\left(f^v\right)&=e^{-4 \varphi_1}\left(\Delta_g^2 f\right)^v-4 e^{-2\left(\varphi_1+\varphi_2\right)}\Big(2r\varphi_1^{\prime \prime}-4 r \varphi_1^{\prime}\left(\varphi_1+\varphi_2\right)^{\prime} \\ & \hspace{7cm}+2 m r(\varphi_1^{\prime})^2+2kr \varphi_1^{\prime} \varphi_2^{\prime}+k \varphi_1^{\prime}\Big)(\Delta_g f)^v.
		\end{align*}
	\end{proposition}
	\begin{proof}
		Let $f:M\longrightarrow\mathbb{R}$ be a smooth function. Then 
		\begin{align*}
			\Delta_{G}(\Delta_{G} f^{v})&=\Delta_{G}\big(e^{-2\varphi _{1}}(\Delta_{g} f)^{v}\big)\\
			&=\Delta_{G}(e^{-2\varphi _{1}})(\Delta_{g} f)^{v}+e^{-2\varphi _{1}}\Delta_{G}\big((\Delta_{g} f)^{v}\big)+2G\big(\tilde{\nabla}(e^{-2\varphi _{1}}),\tilde{\nabla}((\Delta_{g} f)^{v}) \big).
		\end{align*}
		On the other hand, we have
		\begin{align*}
			G\big(\tilde{\nabla}(e^{-2\varphi _{1}}),\tilde{\nabla}((\Delta_{g} f)^{v}) \big) &=G\big(-4\varphi _{1}^{'}e^{-2(\varphi _{1}+\varphi _{2})}\xi, e^{-2\varphi _{1}} (\nabla \Delta_{g}f)^{h} \big)\\
			&=-4\varphi _{1}^{'}e^{-2(2\varphi _{1}+\varphi _{2})} G\big(\xi, (\nabla \Delta_{g}f)^{h} \big)=0, 
		\end{align*}
		\begin{align*}
			\Delta_{G}\big((\Delta_{g} f)^{v}\big)&=e^{-2\varphi _{1}} (\Delta_{g}^2 f)^{v}
		\end{align*}
		and 
		\begin{align*}
			\Delta_{G}\big(e^{-2\varphi _{1}}\big)&=div(\tilde{\nabla}(e^{-2\varphi _{1}}))\\
			&=-4div(\varphi _{1}^{'}e^{-2(\varphi _{1}+\varphi _{2})}\xi)\\
			&=-4\xi.(\varphi _{1}^{'}e^{-2(\varphi _{1}+\varphi _{2})})-4\varphi _{1}^{'}e^{-2(\varphi _{1}+\varphi _{2})} div(\xi)\\
			&=-4 \xi.(\varphi _{1}^{'})e^{-2(\varphi _{1}+\varphi _{2})}-4\varphi _{1}^{'} \xi.(e^{-2(\varphi _{1}+\varphi _{2})})-4\varphi _{1}^{'}e^{-2(\varphi _{1}+\varphi _{2})} div(\xi)\\
			&=-8r\varphi _{1}^{''}e^{-2(\varphi _{1}+\varphi _{2})}+16r\varphi _{1}^{'}(\varphi _{1}+\varphi _{2})^{'}e^{-2(\varphi _{1}+\varphi _{2})}\\
			&\hspace{2cm}- 4\varphi _{1}^{'}\big(2mr\varphi _{1}^{'}+(1+2r\varphi _{2}^{'})k\big)e^{-2(\varphi _{1}+\varphi _{2})}\\
			&=-4e^{-2(\varphi _{1}+\varphi _{2})}\big(2r\varphi _{1}^{''}-4r\varphi _{1}^{'}(\varphi _{1}+\varphi _{2})^{'}+2mr(\varphi _{1}^{'})^{2}+2kr\varphi _{1}^{'}\varphi _{2}^{'}+k\varphi _{1}^{'}\big).    
		\end{align*}
		
		Finally, we get 
		\begin{align*}
			\Delta_{G}^2(f^{v})&= e^{-4\varphi _{1}}(\Delta_{g}^2 f)^{v}-4e^{-2(\varphi _{1}+\varphi _{2})}\Big(2r\varphi _{1}^{''}-4r\varphi _{1}^{'}(\varphi _{1}+\varphi _{2})^{'}\\ & \hspace{7cm}+ 2mr(\varphi _{1}^{'})^{2}+2kr\varphi _{1}^{'}\varphi _{2}^{'}+k\varphi _{1}^{'}\Big)(\Delta_{g} f)^{v}.    
		\end{align*}
		
	\end{proof}
	A straightforward result can be formulated as follows: If $f:M\longrightarrow \mathbb{R}$ is a harmonic function, then $f^v$ is a biharmonic function on $(E,G)$. 
	
	In what follows, we shall make certain choices of the functions $\varphi_1$ and $\varphi_2$ in order to construct (proper) biharmonic functions on $E$ from (proper) biharmonic functions on $M$. 
	
	We remark that the vertical lift of a (proper) biharmonic function on $M$ is a (proper) biharmonic function on $E$ if and only if $$(\mathcal{E})\quad2r\varphi _{1}^{''}-4r\varphi _{1}^{'}(\varphi _{1}+\varphi _{2})^{'}+2mr(\varphi _{1}^{'})^{2}+2rk\varphi _{1}^{'}\varphi _{2}^{'}+k\varphi _{1}^{'}=0.$$ Hence, we shall focus on finding weight functions $\varphi _{1}$ and $\varphi _{2}$ that satisfy the previous equation.  
	
	\begin{corollary}
		Assume that $\varphi_1$ is constant and let $f:M\longrightarrow \mathbb{R}$ be a (proper) biharmonic function on $(M,g)$, then $f^v$ is also a (proper) biharmonic function on $(E,G)$. 
	\end{corollary}
	\begin{proof}
		If $\varphi_1$ is constant, then the equation ($\mathcal{E}$) is trivially verified. Furthermore, if $f$ is a proper biharmonic function, then in particular $f^v$ is biharmonic, and since we have that $\Delta_G f^v=e^{-2\varphi_1} (\Delta_g f)^v$ and $\Delta_g f\neq 0$, then $\Delta_G f^v\neq 0$, which implies that $f^v$ is also a proper biharmonic function. 
	\end{proof}
	
	\begin{example}
		Consider Euclidean space without the origin $M=\mathbb{R}^{n}\setminus\{0\}$ with its Euclidean metric $g_0$. Let $f:\mathbb{R}^{n} \rightarrow \mathbb{R}$ be the polynomial function defined by 
		$$f(x_1,\ldots,x_n) = \Big(\sqrt{x_1^2+\cdots+x_n^2}\Big)^{-1}$$ 
		A straightforward calculation shows that $$\Delta_{g_0}^2f=\frac{3(n-5)(n-3)}{\Big(\sqrt{x_1^2+\cdots+x_n^2}\Big)^{5}}$$
		If $n=3$ or $n=5$, then  $f$ is a biharmonic function on $(M,g_0)$. Further, we prove that $\Delta_{g_0} f=(3-n)f^3$, so if $n=5$, the function is in fact proper biharmonic.
		
		Now, let $\pi:E\longrightarrow M$ be an arbitrary vector bundle and assume either $n=3$ or $n=5$. We endow $E$ with the spherically symmetric metric $G=\pi^{*}g_0\oplus e^{2\varphi_{2}}\pi^{\star}g_0$, where $\varphi_{2}$ is an arbitrary function on $[0,+\infty)$ satisfying the regularity requirements of a weight function of a spherically symmetric metric. Then, the function $f^v:E\longrightarrow \mathbb{R}$ is a biharmonic function on $(E,G)$. Furthermore, when $n=5$, the function $f^v:E\longrightarrow \mathbb{R}$ is a proper biharmonic function.
	\end{example}
	
	Trying to solve equation ($\mathcal{E}$) in various particular cases leads to a differential equation for which the solutions are singular at zero, which contradicts the regularity requirements for the weight functions of a spherically symmetric metric. 
	
	To make this point clearer, we suggest the following example. Remark that equation $(\mathcal{E})$ may be rewritten as 
	$$(\mathcal{E}^\prime)\quad 2r\varphi _{1}^{''}+2r(m-2)(\varphi _{1}^{'})^{2}+2r(k-2)\varphi _{1}^{'}\varphi _{2}^{'}+k\varphi _{1}^{'}=0.$$
	\begin{proposition}\label{exp_vertical}
		In the same notation as before, assume that
		$m=2$ and $k=2$. If, for every biharmonic function $f$ on $M$    the vertical lift $f^v$ is a biharmonic function on $E$, then $\varphi_1$ is constant. 
	\end{proposition}
	\begin{proof}
		Under the conditions $m=2$ and $k=2$, that is, $E$ is a plane bundle over a Riemannian surface, equation $(\mathcal{E}^\prime)$ becomes 
		$$r \varphi_1^{''}+\varphi_1=0.$$
		Set $\psi=\varphi_1'$, the equation becomes a first order differential equation which can be integrated, for $r\neq 0$, as follows  $$\psi(r)=\alpha e^{-\int \frac{dr}{r}},$$
		which gives 
		$$\psi(r)=\frac{\alpha}{r},$$
		for every $r\in ]0,+\infty[$, where $\alpha$ is a real number. Finally we obtain, for every $r\in ]0,+\infty[$,  $\varphi_1(r)=\alpha\ln(r)+\beta$, with $\alpha$ and $\beta$ are real numbers. 
		
		By the definition of spherically symmetric metrics, $\varphi_1$ should necessarily be continuous at $0$, in particular, it should have a finite limit at $0$, therefore $\alpha=0$, otherwise $f$ will have an infinite limit at zero on the right. 
	\end{proof}
	
	
	\section{Biharmonicity of \texorpdfstring{$r$}{r}-radial functions}
	
	
	In this section, we accomplish a study of the biharmonicity of $r$-radial functions. For the rest of this section, let $\alpha:\mathbb{R}^{+}\longrightarrow \mathbb{R}$ be a real valued function, which is smooth on $(0,+\infty)$ and smooth at $r=0$ on the right, and likewise for all its successive derivatives; so the associated $r$-radial function $F:E\longrightarrow \mathbb{R}$ is given by $F(e)=\alpha(r)$ for $e\in E$, with $r=h(e,e)$.
	
	Along with Lemma \ref{Der-fun}, the following lemma will be used in calculations and can be found in \cite{Alb1} along with its proof. We recall it for the sake of completeness. 
	
	\begin{lemma}\label{Der-xi}
		If $X\in \mathfrak{X}(E)$, then $\pi^{\star}D_{X}\xi=X^{V}$, where $X^{V}$ is the vertical component of $X$.
	\end{lemma}
	
	\begin{proposition}\label{Har-radial}(see \cite{Abb-Lak5})
		Let $\alpha:\mathbb{R}^{+}\longrightarrow \mathbb{R}$ be a smooth function on $\mathbb{R}^{+}$, and further let $F:E\longrightarrow \mathbb{R}$ be the $r$-radial function induced from $\alpha$. Then
		{\small
			$$\Delta_{G}F=4e^{-2\varphi_{2}}\Big\{r\alpha''+\big(mr\varphi_{1}'+(k-2)r\varphi_{2}'+\frac{k}{2}\big)\alpha'\Big\}.$$ 
		}
	\end{proposition}
	Using the previous expression of the Laplacian $\Delta_{G}(F)$, we find:
	\begin{proposition}\label{Bihar-radial}
		Let $\alpha:\mathbb{R}^{+}\longrightarrow \mathbb{R}$ be a smooth function on $\mathbb{R}^{+}$, and let $F:E\longrightarrow \mathbb{R}$ be the $r$-radial function induced from $\alpha$, then
		{\small
			\begin{align*}
				\Delta_{G}^2(F)&=-16\varphi_{2}^{'}e^{-4\varphi_{2}}\Big(r\alpha^{''}+(mr\varphi_{1}^{'}+(k-2)r\varphi_{2}^{'}+\frac{k}{2})\alpha^{'}\Big)\alpha^{'}\\
				&+4e^{-2\varphi_{2}}\bigg[e^{-2\varphi_{2}}\bigg(r\alpha^{(3)}+\Big(m\varphi_{1}^{'}+mr\varphi_{1}^{''}+(k-2)\varphi_{2}^{'}+r(k-2)\varphi_{2}^{''}\Big)\\ 
				&\hspace{5cm}+\Big(mr\varphi_{1}^{'}+r(k-2)\varphi_{2}^{'}+\frac{k}{2}+1\Big)\alpha^{''}\bigg)\bigg]div(\xi)\\
				&-8\varphi_{2}^{'}re^{-2\varphi_{2}}\bigg(r\alpha^{(3)}+\Big(m\varphi_{1}^{'}+mr\varphi_{1}^{''}+(k-2)\varphi_{2}^{'}+r(k-2)\varphi_{2}^{''}\Big)\alpha^{'}\\
				&\hspace{6cm}+\Big(mr\varphi_{1}^{'}+r(k-2)\varphi_{2}^{'}+\frac{k}{2}+1\Big)\alpha^{''}\bigg)\\ 
				&+4re^{-2\varphi_{2}}\bigg(\alpha^{(3)}+r\alpha^{(4)}+\Big(2m\varphi_{1}^{''}+mr\varphi_{1}^{(3)}
				+2(k-2)\varphi_{2}^{''}+r(k-2)\varphi_{2}^{(3)}\Big)\alpha^{'}\\
				&\qquad\qquad\qquad\qquad+2\Big(m\varphi_{1}^{'}+mr\varphi_{1}^{''}+(k-2)\varphi_{2}^{'}+r(k-2)\varphi_{2}^{''}\Big)\alpha^{''}\\
				&\qquad\qquad\qquad\qquad+\Big(mr\varphi_{1}^{'}+r(k-2)\varphi_{2}^{'}+\frac{k}{2}+1\Big)\alpha^{(3)}\bigg)\\
				&-8r\varphi_{2}^{'}e^{-6\varphi_{2}}\bigg(r\alpha^{(3)}+\Big(m\varphi_{1}^{'}+mr\varphi_{1}^{''}+(k-2)\varphi_{2}^{'}+r(k-2)\varphi_{2}^{''}\Big)\alpha^{'}\\
				&\qquad\qquad\qquad\qquad+\Big(mr\varphi_{1}^{'}+r(k-2)\varphi_{2}^{'}+\frac{k}{2}+1\Big)\alpha^{''}\bigg).
			\end{align*}
		}
	\end{proposition}
	
	\pagebreak
	\begin{proof}
		We have
		{\small
			\begin{align*}
				\Delta_{G}^2(F)&=4\Delta_{G}\Big(e^{-2\varphi_{2}}\big(r\alpha^{''}+(mr\varphi_{1}^{'}+(k-2)r\varphi_{2}^{'}+\frac{k}{2})\alpha^{'}\big)\Big)\\
				&=4\Big[\Delta_{G}(e^{-2\varphi_{2}})\Big(r\alpha^{''}+(mr\varphi_{1}^{'}+(k-2)r\varphi_{2}^{'}+\frac{k}{2})\alpha^{'}\Big)\\
				&\qquad\qquad +e^{-2\varphi_{2}}\Delta_{G}\big(r\alpha^{''}+(mr\varphi_{1}^{'}+(k-2)r\varphi_{2}^{'}+\frac{k}{2})\alpha^{'}\big)\\
				&\qquad\qquad +2G\big(\tilde{\nabla}(e^{-2\varphi_{2}}),\tilde{\nabla}[r\alpha^{''}+(mr\varphi_{1}^{'}+(k-2)r\varphi_{2}^{'}+\frac{k}{2})\alpha^{'}]\big)\Big].
			\end{align*}
		}
		
		Set
		{\small
			$$L=r\alpha^{''}+(mr\varphi_{1}^{'}+(k-2)r\varphi_{2}^{'}+\frac{k}{2})\alpha^{'}.$$
		}
		We have:
		{\small
			\begin{equation}\label{Cov_phi_2}
				\tilde{\nabla}(e^{-2\varphi_{2}})=-4\varphi_{2}^{'}e^{-4\varphi_{2}}\xi,
			\end{equation}
		}
		and
		{\small
			\begin{align}\label{Cov_L}
				\begin{split}
					\tilde{\nabla}L&=\tilde{\nabla}r.\alpha^{''}+r.\tilde{\nabla}\alpha^{''}+\big(m(\tilde{\nabla}r.\varphi_{1}^{'}+r\tilde{\nabla}\varphi_{1}^{'}) +(k-2)(\tilde{\nabla}r.\varphi_{2}^{'}+r\tilde{\nabla}\varphi_{2}^{'})\big)\alpha^{'}\\
					&\qquad\qquad+\big(mr\varphi_{1}^{'}+r(k-2)\varphi_{2}^{'}+\frac{k}{2}\big)\tilde{\nabla}\alpha^{'}\\
					&=2e^{-2\varphi_{2}}\alpha^{''}\xi+2re^{-2\varphi_{2}}\alpha^{(3)}\xi+2\alpha^{''}e^{-2\varphi_{2}}\big(mr\varphi_{1}^{'}+r(k-2)\varphi_{2}^{'}+\frac{k}{2}\big)\xi \\
					&\qquad \qquad + \big(m(2e^{-2\varphi_{2}}\varphi_{1}^{'}\xi +2r\varphi_{1}^{''}e^{-2\varphi_{2}}\xi)\\
					&\qquad\qquad\qquad\qquad+(k-2)(2\varphi_{2}^{'}e^{-2\varphi_{2}}\xi+2r\varphi_{2}^{''}e^{-2\varphi_{2}}\xi)\big)\alpha^{'}\\
					&=e^{-2\varphi_{2}}\bigg(r\alpha^{(3)}+\Big(m\varphi_{1}^{'}+mr\varphi_{1}^{''}+(k-2)\varphi_{2}^{'}+r(k-2)\varphi_{2}^{''}\Big)\alpha^{'}\\
					&\qquad\qquad\qquad\qquad\qquad\qquad+\Big(mr\varphi_{1}^{'}+r(k-2)\varphi_{2}^{'}+\frac{k}{2}+1\Big)\alpha^{''}\bigg)\xi,
				\end{split}
			\end{align}
		}
		by (\ref{Cov_phi_2}) and (\ref{Cov_L}), we obtain
		{\small
			\begin{align}\label{product_G}
				\begin{split}
					G\big(\tilde{\nabla}(e^{-2\varphi_{2}}),\tilde{\nabla}L\big)&=-4\varphi_{2}^{'}e^{-6\varphi_{2}}G(\xi,\xi)\bigg(r\alpha^{(3)}  +\Big(mr\varphi_{1}^{'}+r(k-2)\varphi_{2}^{'}+\frac{k}{2}+1\Big)\alpha^{''}\Big) \\
					& \qquad\qquad\qquad\qquad\qquad +\Big(m\varphi_{1}^{'}+mr\varphi_{1}^{''}+(k-2)\varphi_{2}^{'}+r(k-2)\varphi_{2}^{''}\bigg)\alpha^{'}\\
					&=-4r\varphi_{2}^{'}e^{-6\varphi_{2}}\bigg(r\alpha^{(3)}+ \Big(mr\varphi_{1}^{'}+r(k-2)\varphi_{2}^{'}+\frac{k}{2}+1\Big)\alpha^{''}\\
					& \qquad\qquad\qquad\qquad +\Big(m\varphi_{1}^{'}+mr\varphi_{1}^{''}+(k-2)\varphi_{2}^{'}+r(k-2)\varphi_{2}^{''}\Big)\alpha^{'}\bigg).
				\end{split}
			\end{align}
		}
		On the other hand, we have
		{\small
			\begin{align*}
				\Delta_{G}L&=div\bigg[
				e^{-2\varphi_{2}}\bigg(r\alpha^{(3)}+\Big(mr\varphi_{1}^{'}+r(k-2)\varphi_{2}^{'}+\frac{k}{2}+1\Big)\alpha^{''}\\
				&\qquad\qquad\qquad\qquad+\Big(m\varphi_{1}^{'}+mr\varphi_{1}^{''}+(k-2)\varphi_{2}^{'}+r(k-2)\varphi_{2}^{''}\Big)\alpha^{'}\bigg)\xi\bigg]\\
				&=\xi.\bigg[e^{-2\varphi_{2}}\bigg(r\alpha^{(3)}+\Big(mr\varphi_{1}^{'}+r(k-2)\varphi_{2}^{'}+\frac{k}{2}+1\Big)\alpha^{''}\\
				&\qquad\qquad\qquad\qquad+\Big(m\varphi_{1}^{'}+mr\varphi_{1}^{''}+(k-2)\varphi_{2}^{'}+r(k-2)\varphi_{2}^{''}\Big)\alpha^{'}\bigg)\bigg]\\
				&\qquad+\bigg[e^{-2\varphi_{2}}\bigg(r\alpha^{(3)}+\Big(mr\varphi_{1}^{'}+r(k-2)\varphi_{2}^{'}+\frac{k}{2}+1\Big)\alpha^{''}\\
				&\qquad\qquad\qquad\qquad+\Big(m\varphi_{1}^{'}+mr\varphi_{1}^{''}+(k-2)\varphi_{2}^{'}+r(k-2)\varphi_{2}^{''}\Big)\alpha^{'}\bigg)\bigg]div(\xi).
			\end{align*}
		}
		Set
		{\small
			\begin{align*}
				S&=e^{-2\varphi_{2}}\bigg(r\alpha^{(3)}+\Big(m\varphi_{1}^{'}+mr\varphi_{1}^{''}+(k-2)\varphi_{2}^{'}+r(k-2)\varphi_{2}^{''}\Big)\alpha^{'}\\
				&\qquad\qquad\qquad\qquad+\Big(mr\varphi_{1}^{'}+r(k-2)\varphi_{2}^{'}+\frac{k}{2}+1\Big)\alpha^{''}\bigg),
			\end{align*}
		}
		thus 
		{\small
			\begin{align*}
				\xi(S)&=\xi.\bigg[e^{-2\varphi_{2}}\bigg(r\alpha^{(3)}+\Big(m\varphi_{1}^{'}+mr\varphi_{1}^{''}+(k-2)\varphi_{2}^{'} +r(k-2)\varphi_{2}^{''}\Big)\alpha^{'}\\
				&\qquad\qquad\qquad\qquad+\Big(mr\varphi_{1}^{'}+r(k-2)\varphi_{2}^{'}+\frac{k}{2}+1\Big)\alpha^{''}\bigg)\bigg]\\
				&=2\xi.(e^{-2\varphi_{2}})\bigg(r\alpha^{(3)}+\Big(m\varphi_{1}^{'}+mr\varphi_{1}^{''}+(k-2)\varphi_{2}^{'}+r(k-2)\varphi_{2}^{''}\Big)\alpha^{'}\\
				&\qquad\qquad\qquad\qquad+\Big(mr\varphi_{1}^{'}+r(k-2)\varphi_{2}^{'}+\frac{k}{2}+1\Big)\alpha^{''}\bigg)\\
				&\qquad\qquad+2e^{-2\varphi_{2}}\bigg(\xi.(r)\alpha^{(3)}+r\xi.(\alpha^{(3)})+\Big(m\xi.(\varphi_{1}^{'})+m\xi.(r)\varphi_{1}^{''}+mr\xi.(\varphi_{1}^{''})\\
				&\qquad\qquad\qquad\qquad\qquad\qquad+(k-2)\xi.(\varphi_{2}^{'})+\xi.(r)(k-2)\varphi_{2}^{''}+r(k-2)\xi(\varphi_{2}^{''})\Big)\alpha^{'}\\
				&\qquad\qquad\qquad\qquad+\Big(m\varphi_{1}^{'}+mr\varphi_{1}^{''}+(k-2)\varphi_{2}^{'}+r(k-2)\varphi_{2}^{''}\Big)\xi.(\alpha^{'})\\
				&\qquad\qquad\qquad\qquad+\Big(m\xi.(r)\varphi_{1}^{'}+mr\xi.(\varphi_{1}^{'})+\xi(r)(k-2)\varphi_{2}^{'}+r(k-2)\xi.(\varphi_{2}^{'})\Big)\alpha^{''}\\
				&\qquad\qquad\qquad\qquad+\Big(mr\varphi_{1}^{'}+r(k-2)\varphi_{2}^{'}+\frac{k}{2}+1\Big)\xi.(\alpha^{''})\bigg)\\
				&=-8\varphi_{2}^{'}re^{-2\varphi_{2}}\bigg(r\alpha^{(3)}+\Big(m\varphi_{1}^{'}+mr\varphi_{1}^{''}+(k-2)\varphi_{2}^{'} +r(k-2)\varphi_{2}^{''}\Big)\alpha^{'}\\
				&\qquad\qquad\qquad\qquad\qquad\qquad+\Big(mr\varphi_{1}^{'}+r(k-2)\varphi_{2}^{'}+\frac{k}{2}+1\Big)\alpha^{''}\bigg)\\
				&\qquad+2e^{-2\varphi_{2}}\bigg(2r\alpha^{(3)}+2r^{2}\alpha^{(4)}+\Big(2mr\varphi_{1}^{''}+2mr\varphi_{1}^{''}+2mr^{2}\varphi_{1}^{(3)}\\
				&\qquad\qquad\qquad\qquad\qquad\qquad\qquad\qquad\qquad+2r(k-2)\varphi_{2}^{''}+2r^{2}(k-2)\varphi_{2}^{(3)}\Big)\alpha^{'}\\
				&\qquad\qquad\qquad\qquad+2r\Big(m\varphi_{1}^{'}+mr\varphi_{1}^{''}+(k-2)\varphi_{2}^{'}+r(k-2)\varphi_{2}^{''}\Big)\alpha^{''}\\
				&\qquad\qquad\qquad\qquad+\Big(2mr\varphi_{1}^{'}+2mr^{2}\varphi_{1}^{''}+2r(k-2)\varphi_{2}^{'}+2r^{2}(k-2)\varphi_{2}^{''}\Big)\alpha^{''}\\
				&\qquad\qquad\qquad\qquad+2r\Big(mr\varphi_{1}^{'}+r(k-2)\varphi_{2}^{'}+\frac{k}{2}+1\Big)\alpha^{(3)}\bigg).
			\end{align*}
		}
		Then
		{\small
			\begin{align}\label{Lap_L}
				\begin{split}
					\Delta_{G}L&=\bigg[
					e^{-2\varphi_{2}}\bigg(r\alpha^{(3)}+\Big(m\varphi_{1}^{'}+mr\varphi_{1}^{''}+(k-2)\varphi_{2}^{'}+r(k-2)\varphi_{2}^{''}\Big)\alpha^{'}\\
					&\qquad\qquad\qquad\qquad+\Big(mr\varphi_{1}^{'}+r(k-2)\varphi_{2}^{'}+\frac{k}{2}+1\Big)\alpha^{''}\bigg)\bigg]div(\xi)\\
					&-8\varphi_{2}^{'}re^{-2\varphi_{2}}\bigg(r\alpha^{(3)}+\Big(m\varphi_{1}^{'}+mr\varphi_{1}^{''}+(k-2)\varphi_{2}^{'}+r(k-2)\varphi_{2}^{''}\Big)\alpha^{'}\\
					&\qquad\qquad\qquad\qquad\qquad+\Big(mr\varphi_{1}^{'}+r(k-2)\varphi_{2}^{'}+\frac{k}{2}+1\Big)\alpha^{''}\bigg)\\
					&+2e^{-2\varphi_{2}}\bigg(2r\alpha^{(3)}+2r^{2}\alpha^{(4)}+\Big(2mr\varphi_{1}^{''}+2mr\varphi_{1}^{''}+2mr^{2}\varphi_{1}^{(3)}\\
					&\qquad\qquad\qquad\qquad\qquad\qquad\qquad+2r(k-2)\varphi_{2}^{''}+2r^{2}(k-2)\varphi_{2}^{(3)}\Big)\alpha^{'}\\
					&\qquad\qquad\qquad+2r\Big(m\varphi_{1}^{'}+mr\varphi_{1}^{''}+(k-2)\varphi_{2}^{'}+r(k-2)\varphi_{2}^{''}\Big)\alpha^{''}\\
					&\qquad\qquad\qquad+\Big(2mr\varphi_{1}^{'}+2mr^{2}\varphi_{1}^{''}+2r(k-2)\varphi_{2}^{'}+2r^{2}(k-2)\varphi_{2}^{''}\Big)\alpha^{''}\\
					&\qquad\qquad\qquad+2r\Big(mr\varphi_{1}^{'}+r(k-2)\varphi_{2}^{'}+\frac{k}{2}+1\Big)\alpha^{(3)}\bigg).
				\end{split}    
			\end{align}
		}
		
		Finally, by (\ref{product_G}) and (\ref{Lap_L}), we obtain the desired equation.
	\end{proof}
	
	\begin{remark}
		It is clear that the expression of the bilaplacian of an $r$-radial function is a highly complex fourth order linear differential equation with variable coefficients. Solving of such a differential equation is a problem which cannot be resolved using the classical approaches to differential equations. We are forced to make certain simplifications in order to give examples of biharmonic functions.
	\end{remark}
	
	\begin{proposition}\label{Sasaki_radial}
		Assume that $\varphi_1=\varphi_2=0$ (i.e., $G$ is the Sasaki metric). Let $F:E\longrightarrow \mathbb{R}$ be an $r$-radial function constructed from a function $\alpha:[0,+\infty[\longrightarrow \mathbb{R}$, then $F$ is biharmonic if and only if $\alpha$ satisfies the linear differential equation 
		\begin{equation}\label{Bi-Rrad}
			2r^{2}\alpha^{(4)}+(3k+4)r\alpha^{(3)}+k(k+2)\alpha^{(2)}=0.
		\end{equation} 
	\end{proposition}
	\begin{proof}
		Using the formula of Proposition \ref{Bihar-radial}, the $r$-radial function $F$ is biharmonic if and only if $\Delta_G^2 F=0$, where
		\begin{align*}\label{Eq1}
			\Delta_{G}^2F&=4
			\bigg(r\alpha^{(3)}+\Big(\frac{k}{2}+1\Big)\alpha^{''}\bigg)\operatorname{div}(\xi)+2\bigg(2r\alpha^{(3)}+2r^{2}\alpha^{(4)}+2r\Big(\frac{k}{2}+1\Big)\alpha^{(3)}\bigg).
		\end{align*}
		Using the expression of the divergence of the tautological vector field $\xi$, we obtain that $\operatorname{div}(\xi)=k$, hence 
		\begin{equation*}
			\Delta_{G}^2F=4r^{2}\alpha^{(4)}+(4kr+4r+2r(k+2))\alpha^{(3)}+2k(k+2)\alpha^{(2)},
		\end{equation*}
		which can be simplified as 
		\begin{equation*}
			\Delta_{G}^2F=4r^{2}\alpha^{(4)}+2(3k+4)r\alpha^{(3)}+2k(k+2)\alpha^{(2)}.
		\end{equation*}
		This gives the desired equation.
	\end{proof}
	
	\begin{corollary}\label{Sasaki_radial2}
		Assume that $\varphi_1=\varphi_2=0$. Let $F:E\longrightarrow \mathbb{R}$ be an $r$-radial function constructed from a function $\alpha:[0,+\infty[\longrightarrow \mathbb{R}$. If $\alpha$ is a polynomial function of degree $1$, then $F$ is biharmonic.
	\end{corollary}
	\begin{proof}
		If $\alpha$ is a polynomial function of degree $1$, then $\alpha^{(2)}=\alpha^{(3)}=\alpha^{(4)}=0$, hence equation (\ref{Bi-Rrad}) is trivially verified.	
	\end{proof}
	
	\begin{example}
		In the case of the Sasaki metric, thanks to the previous corollary,
		we obtain a $2$-parameter family of biharmonic functions. Furthermore, in virtue of Proposition \ref{Har-radial}, an $r$-radial function $F$ constructed from a degree one polynomial function $\alpha:[0,+\infty[\longrightarrow \mathbb{R}$ is harmonic if and only if 
		$\alpha^\prime=0$. Hence by choosing a function $\alpha(r)=a r+b$ with $(a,b)\in \mathbb{R}^2$ and $a\neq 0$, we obtain a proper biharmonic function. Therefore, we obtain a $2$-parameter family of proper biharmonic functions.  
	\end{example} 
	
	In what follows, we shall try to find other solutions of  equation (\ref{Bi-Rrad}). First of all, we make the following change of variable in equation (\ref{Bi-Rrad}). Set $\psi=\alpha^{(2)}$. The equation becomes 
	\begin{equation}\label{Bi-Rrad2}
		2r^{2}\psi^{(2)}+(3k+4)r\psi^{\prime}+k(k+2)\psi=0.
	\end{equation}
	
	We remark that the order of derivation of $\psi$ and the power of $r$ in the equation are somehow related (the order of derivation of $\psi$ and the power of $r$ are the same for the three terms of the equation). This suggests that we can search for solutions of the of form $\psi(r)=\beta r^n$, where $\beta\in \mathbb{R}$ and $n\in \mathbb{Z}$. 
	
	A function of the form $\psi(r)=\beta r^n$, where $\beta\in \mathbb{R}$ and $n\in \mathbb{Z}^*$, is a solution of equation (\ref{Bi-Rrad2}) if and only if 
	\begin{equation}\label{Bi-Rrad3}
		2\beta n(n-1)r^{n}+\beta(3k+4)r^n+k(k+2)\beta r^n=0,
	\end{equation}
	or equivalently
	\begin{equation}\label{Bi-Rrad4}
		2n(n-1)+(3k+4)n+k(k+2)=0,
	\end{equation} 
	which can be rewritten as a quadratic equation in $n$ of the form 
	\begin{equation}\label{Bi-Rrad5}
		2n^2+(3k+2)n+k(k+2)=0.
	\end{equation}
	Hence the problem is reduced to an equation relating the power $n$ and the bundle rank $k$. 
	
	So, by fixing the bundle rank, we will end up solving a quadratic equation in $n$. This will give a procedure for constructing examples of biharmonic functions for different vector bundle ranks. We have the following possibilities for $k$:
	\begin{itemize}
		\item [$\bullet$] In case $k=1$ (i.e., line bundles): equation (\ref{Bi-Rrad5}) becomes $$2n^2+5n+3=0.$$
		The discriminant of the equation is $\Delta=1$, hence the equation has two solutions, $n_1=-1$ and $n_2=\frac{-3}{2}$. Therefore, by choosing the first solution, we have a family of solutions given as $\psi=\frac{\beta}{r}$. These are solutions on $]0,+\infty[$ which cannot be extended by continuity at zero. 
		\item [$\bullet$] In case $k=2$ (i.e., plane bundles): equation (\ref{Bi-Rrad5}) becomes $$2n^2+8n+8=0,$$
		which has a vanishing discriminant, therefore a unique solution $n=-2$. Hence a family of solutions, $\psi(r)=\frac{\beta}{r^2}$, on $]0,+\infty[$ with no extension by continuity at zero.
		\item [$\bullet$] $k\geq 3$: the discriminant of the quadratic equation (\ref{Bi-Rrad5}) is given by 
		\begin{align*}
			\Delta& =(3k+2)^2-8k(k+2)=9k^2+12k+4-8k^2-16k\\
			&=k^2-4k+4=(k-2)^2
		\end{align*}
		thus, we obtain two solutions 
		$$s_1=\frac{-(3k+2)-(k-2)}{4}=-k$$ and
		$$s_2=\frac{-(3k+2)+(k-2)}{4}=\frac{-k}{2}$$
		Hence when $k$ is even, we obtain two families of solutions given by $\beta r^{-k}$ and $\beta r^{-\frac{k}{2}}$, and when $k$ is odd, we obtain a single family solutions given by $\beta r^{-k}$. In all the previous cases, we obtain families of solutions on $]0,+\infty[$ with no extension by continuity to zero.
	\end{itemize}
	
	Finally, we integrate the solutions $\psi$ that we found:
	
	\begin{itemize}
		\item [$\bullet$] For $k=1$, we have $\psi(r)=\beta \frac{1}{r}$, hence $\alpha^\prime(r)=\beta \ln(r)+\gamma$, thus $\alpha(r)=\beta (r\ln(r)-r)+\gamma r+\delta$, where $\delta,\gamma,\beta$ are real constants with $\beta\neq 0$.
		\item [$\bullet$] For $k=2$, we have $\psi(r)=\beta r^{-2}$, thus we obtain $\alpha^\prime(r)=-\frac{\beta}{r}+\gamma$, hence $\alpha(r)=-\beta \ln(r)+\gamma r+\delta$, where $\delta,\gamma,\beta$ are real constants with $\beta\neq 0$.
		\item When $k\geq 3$, we separate the even and odd cases:
		\begin{itemize}
			\item In case $k$ is even, we obtain two families of solutions: $\psi_1(r)=\beta_1 r^{-k}$ and $\psi_2(r)=\beta_2 r^{\frac{-k}{2}}$. Thus, by integration, we find two families 
			\begin{align*}
				\alpha_1^\prime(r)=\frac{\beta_1}{-k+1}r^{1-k}+\gamma_1, \text{ and }
				\alpha_2^\prime(r)=\frac{\beta_2}{-\frac{k}{2}+1}r^{1-\frac{k}{2}}+\gamma_1,
			\end{align*}
			which gives 
			\begin{align*}
				\alpha_1(r)&=\frac{\beta_1}{(1-k)(2-k)}r^{2-k}+\gamma_1 r+\delta_1,\text{ and} \\
				\alpha_1(r)&=\frac{\beta_1}{(1-\frac{k}{2})(2-\frac{k}{2})}r^{2-\frac{k}{2}}+\gamma_2 r+\delta_2,
			\end{align*}
			where $\delta_1,\gamma_1,\beta_1,\delta_2,\gamma_2,\beta_2$ are real constants with $\beta\neq 0$ and $\beta\neq 0$.
			\item In case $k$ is odd: we obtain a family of solutions $\psi(r)=\beta r^{-k}$, hence 
			\[\alpha^\prime(r)=\frac{\beta}{1-k}r^{1-k}+\gamma,\]
			which gives 
			\[\alpha(r)=\frac{\beta}{(1-k)(2-k)}r^{2-k}+\gamma r+\delta,\]
			where $\delta,\gamma,\beta$ are real constants with $\beta\neq 0$.
		\end{itemize}
	\end{itemize}
	
	We remark that all the previous solutions  are defined on $\mathbb{R}^*_+$ with no extension to zero. Therefore, we obtain biharmonic functions on $E^*=E\setminus O$ where $O$ is the zero section. To summarize, by the previous study, we have:
	
	\begin{proposition}\label{Radial_Examples}
		Assume that $\varphi_1=\varphi_2=0$. Let $F:E\longrightarrow \mathbb{R}$ be an $r$-radial function constructed from a function $\alpha:[0,+\infty[\longrightarrow \mathbb{R}$.
		\begin{itemize}
			\item [$\bullet$] If $k=1$, then there exists a family of biharmonic functions on $E^*$ induced by the functions of the form 
			$$\alpha(r)=\beta (r\ln(r)-r)+\gamma r+\delta,$$ 
			where $\delta,\gamma,\beta$ are real constants with $\beta\neq 0$.
			\item [$\bullet$] If $k=2$, then there exists a family of biharmonic functions on $E^*$ induced by the functions of the form 
			$$\alpha(r)=-\beta \ln(r)+\gamma r+\delta,$$ 
			where $\delta,\gamma,\beta$ are real constants with $\beta\neq 0$.
			\item [$\bullet$] If $k\geq 3$, then we have two sub-cases: 
			\begin{itemize}
				\item If $k$ is even, there exist two families of biharmonic functions on $E^*$ induced by functions of the form 
				\begin{align*}
					\alpha_1(r)&=\frac{\beta_1}{(1-k)(2-k)}r^{2-k}+\gamma_1 r+\delta_1,\text{ and} \\
					\alpha_1(r)&=\frac{\beta_1}{(1-\frac{k}{2})(2-\frac{k}{2})}r^{2-\frac{k}{2}}+\gamma_2 r+\delta_2,
				\end{align*}
				where $\delta_1,\gamma_1,\beta_1$ and $\delta_2,\gamma_2,\beta_2$ are real constants with $\beta\neq 0$ and $\beta\neq 0$.
				\item If $k$ is odd, there exists a family of biharmonic functions on $E^*$ induced by functions of the form 
				$$\alpha(r)=\frac{\beta}{(1-k)(2-k)}r^{2-k}+\gamma r+\delta,$$ 
				where $\delta,\gamma,\beta$ are real constants with $\beta\neq 0$.
			\end{itemize}
		\end{itemize} 
	\end{proposition}
	
	

	{\small\bibliography{AbbDouLak}}

	\EditInfo{August 4, 2023}{October 14, 2023}{Ilka Agricola}

\end{document}